\documentclass{amsart}
\usepackage{amsmath, amssymb,epic,graphicx,mathrsfs,enumerate}
\usepackage[all]{xy}

\usepackage{amsthm}
\usepackage{amssymb}
\usepackage{latexsym}
\usepackage{longtable}
\usepackage{epsfig}
\usepackage{amsmath}
\usepackage{hhline}

 \DeclareMathOperator{\perm}{Sym}
 
\DeclareMathOperator{\aut}{Aut}

 \DeclareMathOperator{\psl}{PSL}
\DeclareMathOperator{\pgl}{PGL}

\DeclareMathOperator{\gl}{GL} \DeclareMathOperator{\psu}{PSU}

\DeclareMathOperator{\alt}{Alt}

\renewcommand{\emptyset}{\varnothing}

\newtheorem{thm}{Theorem}

 \newtheorem{lemma}[thm]{Lemma}

\numberwithin{equation}{section}

\renewcommand{\footnote}{\endnote}
\newcommand{\ignore}[1]{}\makeglossary

\begin{document}
\title[Chirality and finite simple groups]{Hypermaps over non-abelian simple groups and strongly symmetric generating sets}
\author[A. Lucchini]{Andrea Lucchini}
\address{Andrea Lucchini, Dipartimento di Matematica ``Tullio Levi-Civita'',
 University of Padova, Via Trieste 63, 35121 Padova, Italy} 
\email{lucchini@math.unipd.it}
         
\author[P. Spiga]{Pablo Spiga}
\address{Pablo Spiga, Dipartimento di Matematica Pura e Applicata,
 University of Milano-Bicocca, Via Cozzi 55, 20126 Milano, Italy} 
\email{pablo.spiga@unimib.it}
	\maketitle

        \begin{abstract}

A generating pair $x, y$ for a group $G$ is said to be \textbf{\textit{symmetric}} if there exists an
automorphism $\varphi_{x,y}$ of $G$ inverting both $x$ and $y$, that is, $x^{\varphi_{x,y}}=x^{-1}$ and $y^{\varphi_{x,y}}=y^{-1}$. Similarly, a group $G$ is said to be \textbf{\textit{strongly symmetric}} if $G$ can be generated with two elements and  if all  generating pairs of $G$ are
symmetric. 

In this paper we classify the finite strongly symmetric non-abelian simple groups. Combinatorially, these are the finite non-abelian simple groups $G$ such that every orientably regular hypermap with
monodromy group $G$ is reflexible.
          \end{abstract}
\section{Introduction}

\bibliographystyle{amsplain}
The aim of this note is to classify the finite strongly symmetric non-abelian simple groups.

\begin{thm}\label{thm:thm}
Let $S$ be a finite non-abelian simple group. Then $S$ is strongly symmetric if and only if $S\cong \psl(2,q)$ for some prime power $q$.
\end{thm}

 Interest on strongly symmetric groups stems from maps and hypermaps, which (roughly speaking) are embeddings of graphs on surfaces, see~\cite{CM}. We  give a brief account on this connection, for more details see~\cite{chira,LeLi}. 

A map on a surface is a decomposition of a closed connected surface into vertices, edges and faces. The vertices and edges of this decomposition form the underlying graph of the map. An automorphism of a map is an automorphism of the underlying graph which can be extended to a homeomorphism of the whole surface. For the definition of hypermaps, which bring us closer to strongly symmetric groups, we need to take a more combinatorial point of view. 

Each map on a orientable surface can be described by two permutations, usually denoted by $R$ and $L$, acting on the set of directed edges (that is, ordered pairs of adjacent vertices) of the underlying graph. The permutation $R$ permutes cyclically the directed edges starting from a given vertex and preserving a chosen orientation of the surface. The permutation $L$ interchanges the end vertices of a given directed edge. The monodromy group of the surface is the group generated by $R$ and $L$ and the map is said to be regular if the monodromy group acts regularly, that is, the identity is the only permutation fixing some element. 

Observe that in a map, we have $L^2=1$. A hypermap is simply given by the combinatorial data $R$ and $L$, where $L$ is not necessarily an involution. Inspired by the topological and geometrical counterpart for maps, a hypermap is said to be reflexible if the assigment $R\mapsto R^{-1}$ and $L\mapsto L^{-1}$ extends to a group automorphism of $\langle R,L\rangle$; otherwise the hypermap is said to be chiral.

It was shown in~\cite[Lemma~7]{chira}, that a finite group $G$ is strongly symmetric if and only if every orientably regular hypermap with monodromy group $G$ is reflexible. In particular, Theorem~\ref{thm:thm} classify the finite non-abelian simple groups $G$ with the property that every orientably regular hypermap with monodromy group $G$ is reflexible.

In our opinion, Theorem~\ref{thm:thm} suggests a natural problem, which in principle should give a measure of how chirality is abundant among regular hypermaps. Let $S$ be a non-abelian simple group and let $\delta(S)$ be the proportion of strongly symmetric generating sets of $S$, that is,
$$\delta(S):=\frac{|\{(x,y)\in S\times S\mid x,y\hbox{ symmetric generating set}\}|}{|\{(x,y)\in S\times S\mid S=\langle x,y\rangle\}|}.$$ 
The closer $\delta(S)$ is to $1$, the more abundant reflexible hypermaps are among  orientably regular hypermaps with monodromy group $S$. Indeed, Theorem~\ref{thm:thm} classifies the groups $S$ attaining $1$. 
We do not have any ``running conjecture'', but we wonder whether statistically it is frequent the case that $\delta(S)<1/2$. Moreover, we wonder whether it is statistically significant the case that $\delta(S)\to 0$ as $|S|\to \infty$, as $S$ runs through a certain family of non-abelian simple groups.
 
\section{Proof of Theorem~$\ref{thm:thm}$}
We start with a preliminary lemma.
\begin{lemma}\label{lem:lem}
Let $n$ be an integer with $n\ge 3$, let $g\in \mathrm{GL}(n,q)$ be a Singer cycle of order $q^n-1$, let $x:=g^{\gcd(n,q-1)}$ and let $a\in \Gamma\mathrm{L}(n,q)$ such that $x^a=zx^{\varepsilon}$, for some $z\in {\bf Z}(\mathrm{GL}(n,q))$ and $\varepsilon\in \{-1,1\}$. Then $z=1$, $\varepsilon =1$ and $a\in \langle g\rangle$.
\end{lemma}
\begin{proof}
Let $e_1,\ldots,e_n$ be the canonical basis of the $n$-dimensional vector space $\mathbb{F}_q^{n}$ of row vectors over the finite field of cardinality $q$. Set $v:=e_1$ and let $P_1$ be the stabilizer in $\mathrm{GL}(n,q)$ of the vector $v$. As $\langle g\rangle$ acts transitively on the set of non-zero vectors of  $\mathbb{F}_q^n$, we may suppose that $a:=bc$, where $b\in P_1$ and $c$ lies in the Galois group $\mathrm{Gal}(\mathbb{F}_q)$ of the field $\mathbb{F}_q$.

As ${\bf Z}(\mathrm{GL}(n,q))$ consists of scalar matrices, we may identify the matrix $z$ with an element in the field $\mathbb{F}_q$. We show that, for every $\ell\in \mathbb{N}$, we have $(vx^\ell)^a=z^\ell vx^{\varepsilon \ell}$. When $\ell=0$, $v^a=v^{bc}=v$, because $b$ and $c$ fix the vector $v=e_1$. When $\ell>0$, we have 
$$(vx^\ell)^a=v^a(x^\ell)^a=v(x^\ell)^a=v(x^a)^\ell=v(zx^{\varepsilon})^\ell=v(z^\ell x^{\varepsilon \ell})=z^\ell vx^{\varepsilon\ell}.$$

Observe that $v,vx,\ldots, vx^{n-1}$ is a basis of $\mathbb{F}_q$ and hence there exists $a_0,a_1,\ldots,a_{n-1}\in\mathbb{F}_q$ with
\begin{equation}\label{lem:1}vx^n=a_0v+a_1vx+\cdots+a_{n-1}vx^{n-1}.
\end{equation}
Now, by applying $a$ on both sides of this equality and using the previous paragraph, we obtain
\begin{equation}\label{lem:2}z^nvx^{\varepsilon n}=a_0v+a_1zvx^{\varepsilon}+\cdots+a_{n-1}z^{n-1}vx^{\varepsilon(n-1)}.
\end{equation}
We let $f(T):=T^n-a_{n-1}T^{n-1}-a_{n-2}T^{n-2}-\cdots-a_1T-a_0\in\mathbb{F}_q[T]$ be the characteristic polynomial of the matrix $x$. Observe that $f(T)$ is irreducible in $\mathbb{F}_q[T]$ because $x=g^{\gcd(n,q-1)}$ acts irreducibly on $\mathbb{F}_q^n$. Let $\lambda\in \mathbb{F}_{q^n}$ be a root of $f(T)$ and observe that $\lambda$ generates the field extension $\mathbb{F}_{q^n}/\mathbb{F}_q$.

We now distinguish the cases, depending on whether $\varepsilon=1$ or $\varepsilon=-1$. Assume $\varepsilon=1$. Using~\eqref{lem:1} and~\eqref{lem:2}, we get $f(T)=f(zT)$. Let $\lambda\in \mathbb{F}_{q^n}$ be a root of $f(T)$ and observe that $\lambda$ generates $\mathbb{F}_{q^n}$ the field extension $\mathbb{F}_{q^n}/\mathbb{F}_q$. Then $z\lambda$ is also a root of $f(T)$ and hence, there exists $i\in \{0,\ldots,n-1\}$, with $\lambda^{q^i}=z\lambda$. This gives $\lambda^{q^i-1}=z\in \mathbb{F}_q$, which implies $i=1$ and $z=1$.

Assume $\varepsilon=-1$. Using~\eqref{lem:1} and~\eqref{lem:2}, we obtain that $z\lambda^{-1}$ is also a root of $f(T)$ and hence, there exists $i\in \{0,\ldots,n-1\}$, with $\lambda^{q^i}=z\lambda^{-1}$. This gives $\lambda^{q^i+1}=z\in \mathbb{F}_q$, which has no solutions because $n\ge 3$.
\end{proof}

\begin{proof}[Proof of Theorem~$\ref{thm:thm}$]
Macbeath has proved in~\cite{mac} that, for every prime power $q$, $\psl(2,q)$ is strongly symmetric; see also~\cite[Proposition~8]{chira}. In particular, for the rest of the proof, we let $S$ be a finite strongly symmetric non-abelian simple group and our task is to show that $S\cong \psl(2,q)$, for some prime power $q$.

Observe that, if $S=\langle s_1, s_2 \rangle$ and $\alpha \in \aut(S)$ inverts both $s_1$ and $s_2$, then $$\alpha^2 \in {\bf C}_{\aut (S)}(s_1) \cap {\bf C}_{\aut (S)}(s_2)={\bf C}_{\aut (S)}({\langle s_1,s_2\rangle})={\bf C}_{\aut (S)}(S)=1.$$ If $\alpha$ is the identity automorphism, then $s_1,s_2$ are involutions and hence $S=\langle s_1,s_2\rangle$ is a dihedral group, contradicting the fact that $S$ is a non-abelian simple group. Therefore $\alpha$ has order 2, that is, $\alpha$ is an involution of $\aut (S)$.
	
In \cite[Theorem~$1.1$]{LeLi}, Leemans and Liebeck have proved that, if $T$ is a finite non-abelian simple group that is not isomorphic to $\alt(7),$ to $\psl(2,q),$ to
$\psl(3,q)$ or to $\psu(3,q),$ then there exist $x, s\in S$ such that the following hold:
\begin{itemize}
	\item[(i)] $T= \langle x, s\rangle;$
	\item[(ii)] $s$ is an involution;
	\item[(iii)] there is no involution $\alpha\in \aut(T)$ such that $x^\alpha=x^{-1},$ $s^\alpha=s.$	
	\end{itemize}
In particular, if $S$ is not isomorphic to  $\alt(7),$  to $\psl(2,q),$ to
$\psl(3,q)$ or to $\psu(3,q),$ then $S$ is not strongly symmetric. In the rest of this proof, we deal with each of these cases separately.

\smallskip

Assume $S=\alt(7)$; in particular, $\aut (S)=\perm(7)$. Let $s_1:=(1,2,3,4,5,6,7)$ and $s_2:=(1,2,3,4,6,7,5)$ and, 
for $i\in \{1,2\},$ let
$\Delta_i:=\{\alpha \in \perm(7) \mid s_i^\alpha=s_i^{-1}\}.$
It can be easily checked that $S=\langle s_1,s_2\rangle$ and 
$$\begin{aligned}\Delta_1=\{&(2,7)(3,6)(2,4), (1,7)(2,6)(3,5), (1,6)(2,5)(3,4), (1,5)(2,4)(6,7),\\ &(1,4)(2,3)(5,7),
 (1,3)(4,7)(5,6), (1,2)(3,7)(4,6)\},\end{aligned}$$ 
$$\begin{aligned}\Delta_2=\{&(2,5)(3,7)(4,6), (1,5)(2,7)(3,6),
(1,7)(2,6)(3,4), (1,6)(2,4)(5,7),\\ &(1,4)(2,3)(5,6), (1,3)(4,5)(6,7), (1,2)(3,5)(4,7)\}.
 \end{aligned}$$ Since $\Delta_1 \cap \Delta_2=\emptyset,$ the generating pair $s_1,s_2$ of $\alt(7)$ witnesses that $\alt(7)$ is not strongly symmetric.

\smallskip

Assume $S=\psl(3,q)$. Since $\psl(3,2)=\psl(2,7),$ we may assume $q>2.$

Let $A:=\aut(S)$, let $d:=\gcd(3,q-1)$ and let $\iota$ be the graph automorphism of $\psl(3,q)$ defined via the inverse-transpose mapping 
$x\mapsto x^{\iota}=(x^{-1})^T$, for every $x\in \psl(3,q)$, 
where $x^{T}$ denotes the transpose of the element $x$ of $\psl(3,q)$. Since $x\in \psl(3,q)$ is not a single matrix, but a coset of the center ${\bf Z}(\mathrm{SL}(3,q))$ in $\mathrm{SL}(3,q)$, there is a slight abuse of notation when we talk about the transpose of the coset $x$. However, since ${\bf Z}(\mathrm{SL}(3,q))$ consists of diagonal matrices, this should cause no confusion.

Next,  let $\Omega_1$ be the set of cyclic subgroups of $S$ generated by a Singer cycle of order $(q^2+q+1)/d$ and, for any $K \in \Omega_1$, let  $$\Delta_K:=\{\alpha\in A\mid \alpha^2=1,\,k^\alpha=k^{-1}\,\,\forall k\in K\}. $$
Observe that the set $\Omega_1$ consists of a single $S$-conjugacy class.  

Let $K\in \Omega_1$, let $k\in K$ be a generator of $K$ and let $\alpha,\beta\in \Delta_K$. Then $k^{\alpha}=k^{-1}=k^\beta$ and hence $\beta^{-1}\alpha\in {\bf C}_{\aut(S)}(k)$. This shows that $\Delta_K\subseteq {\bf C}_{\aut(S)}(k)\alpha$ and that $\Delta_K$ consists of the involutions in ${\bf C}_{\aut(S)}(k)\alpha$.

It is not hard to see that there exists a symmetric matrix in $\gl(3,q)$ having order $q^3-1$. Let $g$ be one of these matrices and let $\bar g$  be its projection in $\pgl(3,q)$. The element $h:=\bar g^d$ generates a subgroup $H\in \Omega_1.$  Since $g$ is symmetric, $g=g^T$ and hence $h^{\iota}=h^{-1}$, that is,  $\iota\in \Delta_H$ and $\Delta_H$ consists of the involutions contained in ${\bf C}_{\aut(S)}(h)\iota$. From Lemma~\ref{lem:lem}, we deduce  that, if $a\in \mathrm{P}\Gamma \mathrm{L}(3,q)$ and $h^a=h^{\varepsilon}$ with $\varepsilon\in \{1,-1\}$, then $\varepsilon=1$ and $a\in\langle \bar g\rangle$. As $\bar{g}^\iota=\bar{g}^{-1}$, we deduce ${\bf C}_{\aut(S)}(h)=\langle \bar g\rangle$ and that $\langle \bar{g},\iota\rangle$ is a dihedral group of order $2(q^2+q+1)$.  Thus
\begin{equation}\label{eq:3}
|\Delta_H|=q^2+q+1.
\end{equation}

 Let $\Omega_2$ be the set of the conjugates of $\iota$ in $A.$  Given $y\in \Omega_2,$ we want to determine the number $\delta_y$ of subgroups $K\in \Omega_1$ with the property that
$y\in \Delta_K.$ 
Consider the bipartite graph having vertex set $\Omega_1\cup\Omega_2$ and having edge set consisting of the pairs $\{K,y\}$ with $K\in \Omega_1, y\in \Omega_2$ and $y\in \Delta_K$. Since $\Omega_1$ and $\Omega_2$ both consist of a single $A$-conjugacy class, the group $A$ acts as a group of automorphisms on our bipartite graph with orbits $\Omega_1$ and $\Omega_2$. Thus, the number of edges of the bipartite graph is
$|\Omega_1||\Delta_H|=|\Omega_2|\delta_y$. Therefore, for every $y\in \Omega_1$, we have
\begin{equation}\label{eq:1}
\delta_y=\frac{|\Omega_1||\Delta_H|}{|\Omega_2|}.
\end{equation}

Let $\omega_H$ be the number of $K\in \Omega_1$ with $\Delta_H\cap \Delta_K\neq \emptyset.$ Clearly
\begin{equation}\label{eq:2}\omega_H \leq \delta_y|\Delta_H|.\end{equation}

We claim that there exists $K\in \Omega_1$ with
$\Delta_H \cap \Delta_K=\emptyset.$ From~\eqref{eq:3},~\eqref{eq:1} and~\eqref{eq:2}, it suffices to show that
$$|\Omega_1| > \omega_H\geq \delta_y |\Delta_H|=\frac{|\Omega_1||\Delta_H|^2}{|\Omega_2|}$$
or, equivalently, that
$$|A:{\bf C}_A(\iota)|=|\Omega_2|> |\Delta_H|^2=(q^2+q+1)^2.
$$ 
Let $G=\text{InnDiag}(S)=\mathrm{PGL}(3,q).$ From~\cite[Chapter~$4$]{GLS3} or~\cite[Proposition~$3.2.11$]{burness}, we have ${\bf C}_G(\iota)=\text{Sp}(2,q)$. Thus
$$|A:{\bf C}_A(\iota)|\geq |G:{\bf C}_G(\iota)|=\frac{(q^3-1)(q^2-1)q^3}{(q^2-1)q}=(q^3-1)q^2.$$
As $q>2,$ it follows $|A:{\bf C}_A(\iota)|>(q^2+q+1)^2$ and our claim is now proved.

Now choose $K=\langle y \rangle \in \Omega_2$ such that $\Delta_H \cap \Delta_K=\emptyset.$ We use the list of the maximal subgroups of $S=\psl(3,q)$, see~\cite[Table~$8.3$]{bhr}. When $q\neq 4,$ ${\bf N}_S(H)$ is a maximal subgroup of $S$ isomorphic to $H:3$ and hence $\langle h,y\rangle=S$. In particular, $h,y$ is a generating pair of $S$ witnessing that $S$ is not strongly symmetric. When $q=4$, we have used the computer algebra system \texttt{magma}~\cite{magma} to show that $\psl(3,4)$ is not strongly symmetric.

\smallskip

Let $S=\psu(3,q)$  and $A=\aut(S).$ Since $\psu(3,2)$ is solvable, $q>2.$ Let $A:=\aut(S)$, let $d:=\gcd(3,q+1)$ and  let $\Omega_1$ be the set of cyclic subgroups of $S$ generated by a Singer cycle of order $(q^2-q+1)/d$ and, for any $K \in \Omega_1$, let  $$\Delta_K:=\{\alpha\in A\mid \alpha^2=1,\,k^\alpha=k^{-1}\,\,\forall k\in K\}. $$
Observe that the set $\Omega_1$ consists of a single $S$-conjugacy class.  

Let $\phi$ be the automorphism of $S$ induced by the Frobenius automorphism $x\mapsto x^q$ of the underlying finite field $\mathbb{F}_{q^2}$ of order $q^2$. It is not hard to see that there exists $g\in {\rm GU}(3,q)$ with $g^\phi=g^{-1}$. Let $\bar g$ be the projection of $g$ in ${\rm PGU}(3,q)$ and let $h:=\bar{g}^d$. Thus $H:=\langle h\rangle\in \Omega_1$ and $\phi\in \Delta_H$. Since ${\bf C}_{{\rm PGU}(3,q)}(\langle \bar g\rangle)=\langle \bar g\rangle$ and since no field automorphism centralizes $H$,  we deduce that $\Delta_H$ is the set of the $q^2-q+1$ involutions in the dihedral group $\langle \bar g, \phi\rangle$ of order $2(q^2-q+1)$ (we are omitting some details here, but these are similar to the arguments in the case of $\psl(3,q)$). In particular, $|\Delta_H|=q^2-q+1$.

Let $\Omega_2$ be the set of the conjugates of $\phi$ in $A.$  Given $y\in \Omega_2,$ we want to determine the number $\delta_y$ of subgroups $K\in \Omega_1$ with 
$y\in \Delta_K.$ Arguing as in the previous paragraph, we deduce that 
$$\delta_y=\frac{|\Omega_1||\Delta_H|}{|\Omega_2|}.$$
Let $\omega_H$ be the number of $K\in \Omega_1$ with $\Delta_H\cap \Delta_K\neq \emptyset.$ Clearly
$$\omega_H \leq \delta_y|\Delta_H|.$$

We claim that there exists $K\in \Omega_1$ with
$\Delta_H \cap \Delta_K=\emptyset.$ It suffices to show that
$$|\Omega_1| > \omega_H\geq \delta_y |\Delta(H)|=\frac{|\Omega_1||\Delta_H|^2}{|\Omega_2|}$$
or, equivalently, that
$$|A:{\bf C}_A(\phi)|=|\Omega_2|> |\Delta_H|^2=(q^2-q+1)^2.
$$ 
Let $G=\text{InnDiag}(S)={\rm PGU}(3,q).$ From~\cite[Chapter~$4$]{GLS3} or~\cite[Proposition~$3.3.15$]{burness}, we have ${\bf C}_G(\phi)=\text{Sp}(2,q)$. It follows $$|A:{\bf C}_A(\phi)|\geq |G:{\bf C}_G(\phi)|=\frac{(q^3+1)(q^2-1)q^3}{(q^2-1)q}= (q^3+1)q^2.$$
It follows $|A:{\bf C}_A(\phi)|>(q^2-q+1)^2$ and our claim is now proved.

Now choose $K=\langle y \rangle \in \Omega_2$ such that $\Delta_H \cap \Delta_K=\emptyset.$ We use the list of the maximal subgroups of $S=\psu(3,q)$, see~\cite[Table~$8.3$]{bhr}. When $q\notin\{3, 5\},$ ${\bf N}_S(H)$ is a maximal subgroup of $S$ isomorphic to $H:3$ and hence $\langle h,y\rangle=S$. In particular, $h,y$ is a generating pair of $S$ witnessing that $S$ is not strongly symmetric. When $q\in\{3,4\}$, we have used the computer algebra system \texttt{magma}~\cite{magma} to show that $\psu(3,3)$ and $\psu(3,5)$ are not strongly symmetric.
\end{proof}


\end{document}